\documentclass[a4paper,11pt]{article}
\usepackage[utf8]{inputenc}

\usepackage{amsmath,amsfonts}
\usepackage{amsthm}
\usepackage{amssymb, amsopn}
\usepackage[english]{babel}
\usepackage{amsfonts}
\usepackage{enumerate}
\usepackage{verbatim}
\usepackage{graphicx}
\usepackage{verbatim}
\usepackage[left=3cm, right=3cm, top=3cm, bottom=3cm]{geometry}
\usepackage{tikz}
\usepackage{mathscinet} 

\usepackage[final,                  
            pdftex,                 
            pdfpagelabels,
            pdfstartview = {FitH},  
            bookmarks,              
            colorlinks,             
            plainpages = false,     
            linktoc=all,            
            linkcolor=black,        
            citecolor=blue,         
            urlcolor=blue,          
            filecolor=black         
            ]{hyperref}

\usetikzlibrary{matrix,arrows,decorations.pathmorphing,decorations.markings}

\theoremstyle{plain}
\newtheorem{theorem}{Theorem}[section]
\newtheorem{lemma}[theorem]{Lemma}
\newtheorem{prop}[theorem]{Proposition}
\newtheorem{cor}[theorem]{Corollary}

\newtheorem{deftheo}[theorem]{Definition/Theorem}

\theoremstyle{definition}
\newtheorem{definition}[theorem]{Definition}

\newtheorem{remark}[theorem]{Remark}
\newtheorem{example}[theorem]{Example}

\theoremstyle{remark}

\renewcommand{\bar}{\overline}

\newcommand{\bbN}{\mathbb{N}}

\newcommand{\bbR}{\mathbb{R}}
\newcommand{\bbZ}{\mathbb{Z}}

\newcommand{\raw}{\rightarrow}

\newcommand{\cug}{\subseteq}

\newcommand{\calA}{\mathcal{A}}
\newcommand{\calB}{\mathcal{B}}

\newcommand{\calD}{\mathcal{D}}

\newcommand{\calG}{\mathcal{G}}

\newcommand{\calM}{\mathcal{M}}

\newcommand{\frh}{\mathfrak{h}}

\newcommand{\rla}{\--}

\newcommand{\isom}{\xrightarrow{\sim}}
\newcommand{\surjra}{\twoheadrightarrow}

\newcommand*\quot[2]{{^{\textstyle #1}\big/_{\textstyle #2}}}

\newcommand{\Address}{
  \bigskip{\footnotesize

  \textsc{Albert-Ludwigs-Universit\"at, Freiburg im Breisgau, Germany}\par\nopagebreak
  \textit{E-mail address}: \texttt{leonardo.patimo@math.uni-freiburg.de}
}}

\DeclareMathOperator{\ima}{Im}

\DeclareMathOperator{\Sym}{Sym}

\DeclareMathOperator{\codim}{codim}
\DeclareMathOperator{\at}{at}
\DeclareMathOperator{\coat}{coat}

\title{A combinatorial formula for the coefficient of $q$ in Kazhdan-Lusztig polynomials}

\author{Leonardo Patimo}

\begin{document}

\maketitle
\begin{abstract}
We propose a combinatorial interpretation of the coefficient of $q$ in Kazhdan-Lusztig polynomials and we prove it for finite simply-laced Weyl groups.
\end{abstract}

For every pair of elements $x,y$ in a Coxeter group $W$, Kazhdan and Lusztig \cite{KL} introduced a polynomial $P_{x,y}(q)\in \bbZ[q]$ known as the Kazhdan-Lusztig polynomial. The definition is elementary and can be given using a recursive formula. Over the course of the last decades, Kazhdan-Lusztig polynomials have played a central role in many different areas of representation theory, from semisimple Lie algebras in characteristic zero to algebraic groups in large characteristic  or to quantum groups at roots of unity (see \cite{W2} for a  history of the subject).

A somewhat elementary but still open problem involving Kazhdan-Lusztig polynomials is the so-called \emph{combinatorial invariance conjecture}, which was proposed by Dyer and Lusztig during the '80s (cf. \cite{Br}). The conjecture states that the Kazhdan-Lusztig polynomial $P_{x,y}(q)$ depends only on the poset structure of the Bruhat interval $[x,y]$. Partial progress has been made towards this conjecture: most notably in \cite{BCM} Brenti, Caselli and Marietti   proved the conjecture in the case $x=e$. 

Even for the coefficient of $q$, which is presumably the simplest coefficient of Kazhdan-Lusztig polynomials to study (for example, the positivity of this coefficient was shown by Tagawa in \cite{Tag}, almost twenty years before the general proof of positivity in \cite{EW1}), a general combinatorial interpretation is missing. 
In this paper we propose a new combinatorial interpretation for the coefficient of $q$. In finite simply-laced type (i.e. if $W$ is a Weyl group of type $ADE$) we are able to show that our combinatorial formula holds. As a consequence we confirm the combinatorial invariance conjecture for the coefficient of $q$ for finite simply-laced groups.

We explain now how this paper is structured. After recalling some background on Coxeter groups and on Kazhdan-Lusztig polynomials in Sections \ref{secCG} and \ref{secKL},  following \cite{Fie2} we introduce the moment graph $\calG$ of a Coxeter group and the related sheaves in Section \ref{secMG}. Moment graphs provide a useful algebraic/combinatorial setup for the study of Kazhdan-Lusztig polynomials (cf. \cite{Lan}). In Section \ref{secqMG} we explain how we can compute the coefficient of $q$ of the polynomial $P_{x,y}(q)$ (which we denote by $q_{x,y}$) by looking at sections of degree $2$ of the structure sheaf $\calA$ on the moment graph. In formulas, we have
\[ q_{x,y}=c_{x,y}-\dim \Gamma_0(\calA,[x,y])^2\]
where $c_{x,y}$ is the number of coatoms of $[x,y]$ and $\Gamma_0(\calA,[x,y])^2$ is the space of sections of degree $2$ of the structure sheaf on $\calG|_{[x,y]}$ considered up to global translations. This formula was actually already obtained by Dyer in \cite{Dye2} with different methods.

Our next step is to give an upper bound on $\dim\Gamma_0(\calA,[x,y])^2$ as follows. For a subset $F$ of the edges in $\calG|_{[x,y]}$ we define an operation of ``diamond closure'' by taking $F^\diamond$ to be the smallest subset of edges such that $F\cug F^\diamond$ and such that whenever we have a ``diamond'' in $\calG|_{[x,y]}$ of the form
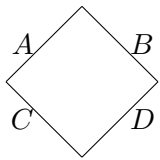
\begin{center}
\begin{tikzpicture}[scale=0.5]
\path[-] 
(0,2) edge node[left] {$A$} (-2,0)
(0,2) edge node[right] {$B$} (2,0)
(0,-2) edge node[left] {$C$} (-2,0)
(0,-2) edge node[right] {$D$} (2,0);
\end{tikzpicture}
\end{center}
 with $A,B\in F^\diamond$, then $C,D\in F^\diamond$ (the orientation of the edges does not matter here). We define $g_{x,y}$ to be the minimal possible cardinality of a set $F$ such that $F^\diamond$ contains all the edges in $\calG|_{[x,y]}$. Then we show:
\begin{equation}\label{introeq}
 \dim \Gamma_0(\calA,[x,y])^2\leq g_{x,y}
\end{equation} 
 
To prove the equality in \eqref{introeq} we need an additional ingredient called the \emph{generalised lifting property}, which is unfortunately available only for finite Weyl groups of type $ADE$. The study of the generalised lifting property is the content of Section \ref{secGLP}.
The coefficient of $q$ in Kazhdan-Lusztig polynomials can be also approached using the related family of $R$-polynomials \cite{KL}.
Recall that if $x<y$ are in $W$ and $s$ is a simple reflection satisfying $xs>x$ and $ys<y$, then we have $y\geq xs$, $x\leq ys$ and the $R$-polynomial $R_{x,y}(q)$ can be obtained using the recursive formula $R_{x,y}(q) = (q-1) R_{x,ys}(q) +q R_{xs,ys}(q)$.

For arbitrary elements $x<y$ in $W$ the existence of such a simple reflection $s$ is not guaranteed. The work of Tsukermann and Williams \cite{TW} in type $A$, extended by Caselli and Sentinelli in arbitrary finite simply-laced type (i.e. in type $D$ and $E$) gives a workaround. They show that we can always find a (not necessarily simple) reflection $t$ such that $y\geq xt\gtrdot x$, $y\gtrdot yt\geq x$ (here  $\gtrdot$ denotes the covering relation for the Bruhat order) and that
\begin{equation}\label{introeq2}
R_{x,y}(q) = (q-1) R_{x,yt}(q) +q R_{xt,yt}(q)
\end{equation}
Such a reflection $t$ is called a \emph{minimal reflection} for $(x,y)$ and it is obtained by taking  a minimal reflection among the reflections satisfying $xt>x$ and $yt<y$ (minimal with respect to the order induced by the dominance order on positive roots, cf. Definition \ref{minrefdef}).
The importance of \eqref{introeq2} lies in the fact that we can use it to relate the coefficient of $q$ for the interval $[x,y]$ to the coefficients of $q$ relative to smaller intervals $[x',y']$, i.e. such that $\ell(y)-\ell(x)>\ell(y')-\ell(x')$.
This finally enables us to prove by induction in Section \ref{secProof} the combinatorial formula
\[q_{x,y}=c_{x,y}-g_{x,y}.\]

We remark that as an intermediate (and crucial) step in our proof we generalise to minimal reflections the following well-known property of simple reflections:
let  $x<y$  and assume $t$ is a minimal reflection for $(x,y)$. Then there exists a maximal chain $x\lessdot z_1\lessdot z_2\lessdot \ldots\lessdot z_{\ell(y)-\ell(x)-2}<yt$ such that $z_it\gtrdot z_i$ for all $i$. 

\section{Preliminaries on Coxeter groups}\label{secCG}

We refer to \cite{Hum} for background material on Coxeter groups.

Let $(W,S)$ be a Coxeter system. For $s,t\in S$, let $m_{st}\in \bbN$ denote the order of $st$. We say that $W$ is \emph{simply-laced} if $m_{st}\leq 3$ for any $s,t\in S$.

 We denote by $\ell$ the length function and by $\geq$ the Bruhat order.
Let $T=\bigcup_{w\in W} wSw^{-1}$ denote the set of reflections. For $x,y\in W$ let 
\[D(y) = \{ t \in T \mid yt <y\}\]
\[A(x) = \{ t \in T \mid xt > x\}\]
\[AD(x,y) = A(x) \cap D(y).\]
We have $|D(y)|=\ell(y)$, hence $|AD(x,y)|\geq \ell(y)-\ell(x)$ and $AD(x,y)$ is not empty for any $x<y$.

 We denote by $\lessdot$ the covering relation of the Bruhat order, i.e. $x\lessdot y$ if $x<y$ and $\ell(x)+1=\ell(y)$. We denote by $[x,y]:=\{z\in W \mid x\leq z\leq y\}$ the Bruhat interval. The intervals $[x,y)$ and $(x,y]$ are similarly defined. 
We denote by
\[\at(x,y) := \{ z \in W | x\lessdot z \leq y\}\] 
\[\coat(x,y) := \{ z \in W | x\leq z \lessdot y\}\] 
the set of \emph{atoms} and \emph{coatoms} of $[x,y]$. We also define 
\[\at^T(x,y) := \left\{ x^{-1}z \;\middle|\; z \in \at(x,y)\right\} \cug T\] 
\[\coat^T(x,y) := \left\{ z^{-1}y  \;\middle|\; z \in \coat(x,y) \right\}\cug T. \]

We fix $\frh$ to be a finite dimensional vector space over $\bbR$ with the property that there exist linearly independent subsets 
\[\{\alpha_s\}_{s\in S}\cug \frh^*\quad\text{ and }\quad \{\alpha_s^\vee\}_{s\in S}\cug \frh\]
such that $\alpha_s(\alpha_t^\vee)=-2 \cos (\pi/m_{st})$. We further assume that $\frh$ is of minimal possible dimension among representations satisfying the above properties.

Then the assignment $s(v):=v-\alpha_s(v)\alpha_s^\vee$ for any $s\in S$ defines a representation of $W$ on $\frh$. Notice that if $W$ is finite, then $\dim \frh = |S|$ and $\frh$ is the geometric representation of $W$ (cf. \cite[\S 5.3]{Hum}).
As shown in \cite[Proposition 2.1]{S4} the representation $\frh$ is reflection faithful, i.e. the  subset of fixed points of $t\in W$ form a hyperplane in $\frh$ if and only if $t\in T$.

We call $\{\alpha_s\}_{s\in S}$ and $\{\alpha_s^\vee\}_{s\in S}$ the set of simple roots and simple coroots respectively. 
Let $\Phi= \{ w(\alpha_s) \mid w \in W, s\in S\}\cug \frh^*$ denote the set of roots. 
Let $\Phi^+\cug \Phi$ be the subset of positive roots, i.e. the subset of roots which are a positive linear combination of simple roots. 

For a reflection $t\in T$ there exists $w \in W$  and $s \in S$  with $ws > w$ such that $t=wsw^{-1}$. Then we set $\alpha_t:=w(\alpha_s) \in \Phi^+$. Notice that $\alpha_t$ vanishes on the hyperplane fixed by $t$. The root $\alpha_t$ is well-defined and, if $r\neq t$, the roots $\alpha_r$ and $\alpha_t$ are linearly independent.

We introduce a partial order $\succ$ on $\Phi^+$ by setting $\alpha\succ \beta$ if $\alpha-\beta$ is a positive linear combination of simple roots.
This also induces a partial order $\succ$ on $T$ where, for $r,t\in T$, we say $r\succ t$ if $\alpha_r \succ \alpha_t$.

The following simple lemma describes Bruhat intervals of length $2$ and it is applied several times in this paper.
\begin{lemma}\label{squares}
Let $x,y\in W$ and assume $x< y$ and $\ell(y)-\ell(x)=2$. Then there exists $r,t\in A(x)$ with $r\neq t$ such that $[x,y]=\{x,xr,xt,y\}$. 

Moreover, if $W$ is simply-laced, then $y=xtr$ or $y=xrt$.
\end{lemma}
\begin{proof}
The first part is \cite[Lemma 2.7.3]{BjBr}.

Assume now that $W$ is simply-laced. By \cite[Lemma  3.1]{Dye}, the reflection subgroup $W'=\langle r,t,rx^{-1}y,tx^{-1}y\rangle$ is isomorphic to a simply-laced Coxeter subgroup of rank $2$ (i.e. $W'$ is of type $A_1\times A_1$ or of type $A_2$) and, by \cite[Proposition  2.1]{Dye} the interval $[x,y]$ is isomorphic to a Bruhat interval in $W'$. It is then enough to prove the claim for Coxeter groups of type $A_1\times A_1$ or of type $A_2$, and in these cases it is a trivial check.
\end{proof}

The following Lemma is needed in Section \ref{secGLP} (and it is crucial in Caselli and Sentinelli's proof of the generalised lifting property).

\begin{lemma}[{\cite[Proposition 2.3]{CS}}]\label{23}
	Let $t_1,t_2\in T$.
	\begin{enumerate}[i)]
		\item If there exists $x \in W$ such that $t_1\in D(x)$ and $t_2t_1t_2\in A(xt_2)$, then $t_1\in D(t_2)$.
		\item If there exists $x \in W$ such that $t_1\in D(x)$ and $t_2t_1t_2\in D(xt_2)$, then $t_1\in A(t_2)$.
	\end{enumerate}
\end{lemma}

Assume $W$ to be simply-laced.
Let $r,t\in T$ with $r\neq t$. If $r \in D(t)$, then $rt\neq tr$ \cite[Corollary 2.4]{CS}, hence $r,t$ generate a Coxeter group of type $A_2$.
 Moreover,  from \cite[Corollary 3.4]{CS} we have
\begin{equation} \label{Dimplies}
 r\in D(t)\setminus \{t\} \implies  \alpha_r+\alpha_{rtr}=\alpha_t\text{,\; and so }\;r\prec t.
\end{equation}
\section{Kazhdan-Lusztig polynomial and \texorpdfstring{$R$}{R}-polynomials}\label{secKL}


In \cite{KL} Kazhdan-Lusztig polynomials are originally introduced as coefficients of a certain canonical basis in the Iwahori-Hecke algebra. Since the Iwahori-Hecke algebra does not play any direct role in this work, we follow \cite[\S 5.1]{BjBr} and we give instead an equivalent definition which emphasises the relations of Kazhdan-Lusztig polynomials with the $R$-polynomials.

\begin{deftheo}\label{Rpol}
Let $W$ be a Coxeter group. There exists a unique family of polynomials with integral coefficients $\{R_{x,y}(q)\}_{x,y\in W}$, called \emph{$R$-polynomials}, satisfying the following conditions:
\begin{enumerate}[i)]
\item $R_{x,x}(q)=1$ 
\item $R_{x,y}(q)=0$ if $x\not\leq y$, 
\item for any $s\in S$ such that $ys<y$ we have 
\[R_{x,y}(q)=\begin{cases}
R_{xs,ys}(q)& \text{ if }xs<x,\\
qR_{xs,ys}(q)+(q-1)R_{x,ys}(q) & \text{ if }xs>x.\\
\end{cases}\]
\end{enumerate}
\end{deftheo}

\begin{deftheo}
Let $W$ be a Coxeter group. There exists a unique family of polynomials with integral coefficients $\{P_{x,y}(q)\}_{x,y\in W}$, called \emph{Kazhdan-Lusztig polynomials}, satisfying the following conditions:
\begin{enumerate}[i)]
\item $P_{x,x}(q)=1$,
\item $P_{x,y}(q)=0$ if $x\not\leq y$,
\item $\deg(P_{x,y}(q))\leq \frac12  (\ell(y)-\ell(x)-1)$ if $x<y$,
\item for any $x\leq y$ we have
\begin{equation}\label{KLandR}
q^{\ell(y)-\ell(x)}P_{x,y}(q^{-1}) = \sum_{z\in[x,y]} R_{x,z}(q)P_{z,y}(q).
\end{equation}
\end{enumerate}
\end{deftheo}

It is  an immediate consequence of the definitions that for any $x\leq y$ the polynomial $R_{x,y}(q)$ is monic of degree $\ell(y)-\ell(x)$ and the polynomial $P_{x,y}(q)$ has constant term equal to $1$.

We introduce some further notation. For $x\leq y$ let $q_{x,y},d_{x,y}\in \bbZ$ be such that
\[ P_{x,y}(q) = 1 + q_{x,y}q + \text{``higher terms in $q$''},\]
\[ R_{x,y}(q) = q^{\ell(y)-\ell(x)}- d_{x,y}q^{\ell(y)-\ell(x)-1} + \text{``lower terms in $q$''}.\]

\begin{lemma}\label{drec}
Let $W$ be a Coxeter group. Then: 
\begin{enumerate}[i)]
\item $d_{x,x}=0$ for any $x\in W$,
\item for any $x\leq y$ and for any $s\in S$ such that $ys<y$ we have
\[
d_{x,y} = \begin{cases}
d_{xs,ys} & \text{ if }xs<x,\\
d_{x,ys} +1 & \text{ if }xs>x \text{ and } xs\not\leq ys, \\
d_{x,ys}  & \text{ if }xs>x \text{ and } xs\leq ys. \\
\end{cases}
\]
\end{enumerate}
\end{lemma}
\begin{proof}
This is an immediate consequence of Definition/Theorem \ref{Rpol}.
\end{proof}
 
The only terms contributing to the coefficient of $q^{\ell(y)-\ell(x)-1}$ in the RHS of \eqref{KLandR} are $z=y$ and $z\in \coat(x,y)$.
If $c_{x,y}:= |\coat(x,y)|$ we obtain
\begin{equation}\label{q=c-d}
q_{x,y}=-d_{x,y}+\sum_{z\in \coat(x,y)} 1=c_{x,y}-d_{x,y}. 
\end{equation}
The main purpose of the next sections is to provide a combinatorial interpretation of the term $d_{x,y}$.
\section{The moment graph of Coxeter groups}\label{secMG}

We recall the definition of moment graphs and of the related sheaves from \cite{Fie2}.
Recall that $\frh$ is the reflection faithful representation of $W$ introduced in Section \ref{secCG}.

\begin{definition}
The \emph{moment graph} $\calG:=\calG(W,\frh)$ of a Coxeter group $W$ is a labelled directed graph defined as follows.
The set of vertices is the set of elements of $W$. Two vertices $v,w$ are connected by an arrow $v\raw w$ if $w>v$ and there exists a reflection $t\in T$ such that $w=vt$. We label this edge by $\alpha_t\in \frh^*$, where $\alpha_t$ is the positive root corresponding to $t$.\footnote{Note that this differs from \cite{Fie2} since we use right multiplication by reflections (instead of left) to label the edges of $\calG$.}
\end{definition}

The moment graph of Coxeter group is also often denoted Bruhat graph (cf. \cite[Definition 2.1.1]{BjBr}).
For a subset $A\cug W$, we denote by $\calG|_{A}$ the full subgraph of $\calG$ whose vertices are the elements of $A$.

We denote by $R=\Sym_\bbR(\frh^*)$ the symmetric algebra of $\frh^*$. We regard $R$ as a graded algebra with $\frh^*$ sitting in degree $2$.

\begin{definition}
A sheaf $\calM$ on the moment graph of $W$ is given by
\begin{itemize}
\item a graded $R$-module $\calM_x$ for any $x\in W$,
\item for any arrow $x\raw xt$ in $\calG$ a graded $R$-module $\calM_{x\raw xt}$ such that $\alpha_t\cdot \calM_{x\raw xt}=0$,
\item for any arrow $x\raw xt$ in $ \calG$ two morphisms  of graded $R$-modules $\pi_{x,xt}:\calM_x\raw \calM_{x\raw xt}$ and $\pi_{xt,x}:\calM_{xt}\raw \calM_{x\raw xt}$.
\end{itemize}
\end{definition}



If $A\cug W$, the space of sections of a sheaf $\calM$ over $A$ is 
\[\Gamma(\calM,A):=\left\{(m_x)\in \prod_{x\in A}\calM_x\mid\begin{array}{c} \pi_{x,xt}(m_x)=\pi_{xt,x}(m_{xt})\\ \text{ for all } x\in A\text{ and } t \in T\text{ such that }xt\in A\end{array}\right\}.\]
The \emph{space of global sections} of a sheaf $\calM$ is $\Gamma(\calM):=\Gamma(\calM,W)$.
We simply write $\Gamma(\calM,>x)$ for $\Gamma(\calM,\{w \in W \mid w>x\})$ and similarly for $\geq x,\leq x$ and $<x$.

\begin{example}\label{strsheaf}
A trivial example of a sheaf on the moment graph is the structure sheaf $\calA$. It is defined by setting  $\calA_x:= R$ for any $x\in W$ and $\calA_{x\raw xt}:=R/(\alpha_t)$ for any arrow $x\raw xt$ in $\calG$. The maps $\pi_{x,tx}$ and $\pi_{xt,x}$ are the natural projections $R\twoheadrightarrow R/(\alpha_t)$. Notice that the orientation of the arrows in $\calG$ does not matter in the definition of $\calA$.\end{example}

Let $y\in W$. We quickly recall the construction of the canonical sheaf $\calB(y)$ from \cite{BMP}. We start by setting $\calB(y)_y=R$ and $\calB(y)_z= 0$ for any $z\not \leq y$.

Fix $x \leq y$ and assume that $\calB(y)_z$ is already defined for all $z>x$. Then for any $t\in A(x)$ we define $\calB(y)_{x\raw xt}:= \calB(y)_{xt}/(\alpha_t)$ and $\pi_{xt,x}:\calB(y)_{xt}\raw \calB(y)_{x\raw xt}$ to be the natural projection.

 Let 
\[\calB(y)_{\delta x}:= \ima\left(p:\Gamma(\calB(y),>x)\raw \bigoplus_{t\in A(x)}  \calB(y)_{x\raw xt}\right)\] 
where $p$ is the restriction map.
Define $\calB(y)_x$ to be the projective cover  of $\calB(y)_{\delta x}$ as a  graded $R$-module. Notice that $\calB(y)_x$ is a free $R$-module. Finally, for any $t\in A(x)$, let the morphism $\pi_{x,xt}$ be defined as the following composition \[\calB(y)_x\raw \calB(y)_{\delta x} \hookrightarrow \bigoplus_{t\in A(x)}  \calB(y)_{x\raw xt}\surjra \calB(y)_{x\raw xt}.\]  

\begin{remark}
Assume $W$ is the Weyl group of a complex algebraic group $G$ with maximal torus $T$. 
Let $X$ be the flag variety of $G$ and for $w\in W$ we denote by $X_w$ the corresponding Schubert variety. Then we can use the sheaves $\calA$ and $\calB(w)$ to compute the $T$-equivariant cohomology and $T$-equivariant intersection cohomology of $X_w$. More precisely, we have isomorphisms of $R$-modules (cf. \cite[Theorem 1.2.2]{GKM} and \cite[Theorem 1.5]{BMP})
\[\Gamma(\calA,\leq w)\cong H_T(X_w,\bbR)\quad\text{ and }\quad\Gamma(\calB(w))\cong IH_T(X_w,\bbR).\]
\end{remark}

If $V$ is a graded vector space, we denote by $V^i$ its graded component of degree $i$. We regard $\bbR$ as a $R$-module via the isomorphism $\bbR\cong R/R^+$, where $R^+$ denotes the ideal of elements of positive degree.
\begin{theorem}\label{KLconj}
The Poincar\'e polynomial of $\bar{\calB(y)_x}=\calB(y)_x\otimes_R \bbR$ is the Kazhdan-Lusztig polynomial $P_{x,y}(q^{2})$, i.e. 
 we have 
\[ P_{x,y}(q) = \sum_{i\geq 0}\dim \bar{\calB(y)_x}^{2i} q^i.\] 
\end{theorem}

The proof of Theorem \ref{KLconj} is due to Braden and MacPherson \cite{BMP} for finite and affine Weyl groups. Their proof employs geometric techniques and it ultimately relies on Saito's theory of Hodge modules. For arbitrary Coxeter groups the Theorem follows by combining the work of Fiebig \cite{Fie2} (which links sheaves on the moment graphs to Soergel bimodules) and the work of Elias and Williamson: in \cite{EW1} they prove Theorem \ref{KLconj} in the setting of Soergel bimodules by developing an algebraic version of Hodge theory.

\begin{remark}
A remarkable consequence of Theorem \ref{KLconj} is that the polynomial $P_{x,y}(q)$ only depends on the subgraph $\calG|_{[x,y]}$. 
On the other hand the poset structure of $[x,y]$ determines $\calG|_{[x,y]}$ as an unlabelled directed graph \cite{Dye}. 
Thus, the missing step in  the combinatorial invariance conjecture is to show that the labels in $\calG$ are actually a  superfluous piece of data for the computations of the  Kazhdan-Lusztig polynomials.
\end{remark}

\section{Moment graphs and the coefficient of \texorpdfstring{$q$}{q}}\label{secqMG}

For a graded $R$-module $M$ we denote by $M^i$ its graded component of degree $i$ and we denote by $M[i]$ its grading shift, i.e. $M[i]^k=M^{i+k}$.

If $x\leq y$ the constant term of $P_{x,y}(q)$ is $1$, hence the $R$-module $\calB(y)_x$ contains a unique summand isomorphic to $R$: this can be obtained as the $R$-submodule generated by any non-zero element in degree $0$.
Moreover, from Theorem \ref{KLconj} it follows that $\calB(y)_x$ contains the summand $R[-2]$ with multiplicity $q_{x,y}$.

Fix $y\in W$. For every $x\leq y$ we choose $1_x\in R^0\cug \calB(y)_x^0$ compatibly so that if $xt\leq y$ for some $t\in T$, then $1_x$ and $1_{xt}$ are identified via the maps $\pi_{x,xt}$ and $\pi_{xt,x}$.
For every $x$ the vector space $\Gamma(\calB(y),>x)^0$ is one dimensional and it is generated by the section $(1_z)_{z\in(x,y]}$.


We define $\Gamma_0(\calB(y),>x)$ to be the subspace of sections $(f_z)_{z\in(x,y]}\in \Gamma(\calB(y),>x)$ such that $f_y=0$. Let $i_0:\Gamma_0(\calB(y),>x)\hookrightarrow \Gamma(\calB(y),>x)$ be the inclusion.

Consider the following diagram:
\begin{center}
\begin{tikzpicture}
\node (a) at (4,2) {$\Gamma(\calB(y),>x)$};
\node (b) at (4,0) {$\calB(y)_{\delta x}$};
\node (c) at (0,0) {$R$};
\node (d) at (0,2) {$\Gamma_0(\calB(y),>x)$};
\path[->>] 
(a) edge node[right] {$p$} (b);
\path[->]
(c) edge node[above] {$j$} (b)
(d) edge node[above] {$p_0$} (b);
\path[right hook->] (d) edge node[above] {$i_0$} (a);
\end{tikzpicture}
\end{center}
where $j$ is the morphism of $R$-modules defined by $j(1)=p((1_z)_{z\in (x,y]})$. Let $p_0:=p\circ i_0$.

Recall that if $M$ is a graded $R$-module, then the projective cover of $M$ is isomorphic to $R\otimes_{\bbR} \bar{M}$, with $\bar{M}=M\otimes_R \bbR$.
Therefore it is  easy to see that
\[ q_{x,y} =\dim \left(\bar{\calB(y)_{\delta x}}^2\right)=\codim \left(\ima(j)^2 \cug \calB(y)_{\delta x}^2\right)=\dim\left(\calB(y)_{\delta x}/\ima(j)\right)^2.\]
We have
\[\quot{\Gamma_0(\calB(y),>x)}{p_0^{-1}(\ima(j))}\overset{\bar{i_0}}{\lhook\joinrel\longrightarrow} \quot{\Gamma(\calB(y),>x)}{p^{-1}(\ima(j))}\isom\quot{\calB(y)_{\delta x}}{\ima(j)}.\]

The map $\bar{i_0}$ is also surjective: in fact, if $f=(f_z)_{z\in (x,y]}\in \Gamma(\calB(y),>x)$, we can write
\[ f = (f-\bar{f}) + \bar{f}\]
where $\bar{f}$ is the constant section $\bar{f}:=(f_y)_{z\in (x,y]}\in p^{-1}(j(f_y))$. It follows that
\begin{equation}\label{q=G-p}
q_{x,y} = \dim \Gamma_0(\calB(y),>x)^2 -\dim p_0^{-1}(\ima(j)^2).
\end{equation}


The first term in \eqref{q=G-p} has an immediate combinatorial interpretation.
\begin{lemma}\label{coatoms}
Let $x,y\in W$ with $x<y$.
\begin{enumerate}[i)]
	\item If $\ell(y)-\ell(x)\geq 2$, then $\Gamma_0(\calB(y),>x)\cong \Gamma_0(\calB(y),\geq x)$.
	\item The dimension of $\Gamma_0(\calB(y),\geq x)^2$ is equal to the number of coatoms $c_{x,y}= |\coat(x,y)|$ of $[x,y]$.
\end{enumerate}

\end{lemma}
\begin{proof}
First assume $\ell(y)-\ell(x)\geq 2$. From Lemma \ref{squares} it follows that  there are (at least) two reflections $t_1,t_2\in A(x)$ such that $xt_1,xt_2\leq y$. 
We claim that the morphism $j: R \raw \calB(y)_{\delta x}$ is injective in degrees $0$ and $2$. Let $f\in R$  such that $j(f)=0$, so we have $\alpha_{t_1}\mid f$ and $\alpha_{t_2}\mid f$. Since $\alpha_{t_1}$ and $\alpha_{t_2}$ are linearly independent we conclude that $f=0$ or $\deg(f)\geq 4$. 

Since $\calB(y)_x$ is the projective cover of $\calB(y)_{\delta x}$, it  follows that the map $\calB(y)_{x}\raw \calB(y)_{\delta x}$ is bijective in degrees $0$ and $2$.
Hence if $\ell(y)-\ell(x)\geq 2$, every section in $\Gamma(\calB(y),>x)^2$ extends uniquely to a section in $\Gamma(\calB(y),\geq x)^2$ and i) follows. By repeating this argument we see that every section in $\Gamma_0(\calB(y),\coat(x,y) \cup \{y\})^2$ extends uniquely to a section in $\Gamma_0(\calB(y),\geq x)^2$. Moreover, we have 
\[\Gamma_0(\calB(y),\geq x)^2 \cong\Gamma_0(\calB(y),\coat(x,y) \cup \{y\})^2\cong \bigoplus_{z\in \coat(x,y)} \Gamma_0(\calB(y),\geq z)^2.\]

Hence it is enough to consider the case $\ell(y)-\ell(x)=1$. Then $\calB(y)_x\cong R$ and a section in $\Gamma_{0}(\calB(y),\geq x)$ can be thought of as a polynomial $f\in \calB(y)_x$ such that $\alpha_{x^{-1}y}\mid f$. It follows that $\dim \Gamma_0(\calB(y),\geq x)^2=1$.
\end{proof}

 We give now a more insightful description of the vector space $p_0^{-1}(\ima(j)^2)$.

%
%
%

\begin{prop}
	Let $x,y\in W$ with $x<y$ and $\ell(y)-\ell(x)\geq 2$. We have:
\[ \Gamma_0(\calA,[x,y])^2\cong  p_0^{-1}(\ima(j))^2.\]
\end{prop}
\begin{proof}
Since $\ell(y)-\ell(x)\geq 2$, the same argument as in the proof of Lemma \ref{coatoms} shows that the restriction map $\Gamma_0(\calA,[x,y])^2\raw \Gamma_0(\calA,(x,y])^2$ is injective.
Moreover, the map $ \Gamma_0(\calA,(x,y])\raw p_0^{-1}(\ima(j))$ induced by the inclusions $R\hookrightarrow \calB(y)_z$ for $z\in (x,y]$ is clearly injective.

It remains to show that the composition map $\varphi:\Gamma_0(\calA,[x,y])^2\raw p_0^{-1}(\ima(j))^2$ is surjective.
Let $f=(f_z)_{z\in (x,y]}\in \Gamma_0(\calB(y),>x)^2$ and assume that $f\in p_0^{-1}(\ima(j)^2)$. For every $z\in (x,y]$, we have $\calB(y)_z=R\oplus R[-2]^{q_{z,y}}\oplus \ldots$, so we can write $f_z=f_{z,0}+f_{z,2}$ with $f_{z,0}\in R$ and $f_{z,2}\in R[-2]^{q_{z,y}}$. 

Let $a$ be an atom in $[x,y]$ with $t=x^{-1}a\in T$. The image of the map 
\[R\raw \calB(y)_{x\raw a}=\calB(y)_a/(\alpha_t)=R/(\alpha_t)\oplus R/(\alpha_t)[-2]^{q_{a,y}}\oplus \ldots \]
is contained in the summand $R/(\alpha_t)$ for degree reasons. Then the assumption $f\in p_0^{-1}(\ima(j)^2)$ forces to have $f_{a,2}=0$.

Fix $z\in (x,y]$ and assume that $f_{w,2}=0$ for all $w<z$. If $z$ is not an atom in $[x,y]$, we can choose $t\in D(z)$ such that $zt\in (x,z]$. Since $f_{zt}=f_{zt,0}$ the same argument as above shows $f_{z,2}=0$. Hence by induction we obtain that $f_{z,2}=0$ for any $z\in (x,y]$ or, equivalently, that $f_z$ is contained in the summand $R\cug \calB(y)_z$. 

This implies that every section $f\in p_0^{-1}(\ima(j)^2)$ can be thought as a section of the structure sheaf $\calA$, hence it is in the image of $\varphi$.
\end{proof}

Notice that, after a trivial check in the case $\ell(y)-\ell(x)\leq 1$, we can rewrite \eqref{q=G-p} as
\[ q_{x,y} = \dim \Gamma_0(\calB(y),\geq x)^2 - \dim \Gamma_0(\calA,[x,y])^2= c_{x,y} -\dim \Gamma_0(\calA,[x,y])^2.\]
As a consequence of \eqref{q=c-d} and Lemma \ref{coatoms}  we obtain $\dim \Gamma_0(\calA,[x,y])^2=d_{x,y}$.

We identify $R^2$ with the vector space $\frh^*$. Under this identification 
$\Gamma_0(\calA,[x,y])$ corresponds to the vector space
\[ V_{x,y} := \left\{ (v_z)_{z\in [x,y]}\in \bigoplus_{z\in [x,y]} \frh^*\;\middle|\; 
\begin{tabular}{c}
$v_y=0$ and \\ 
$v_{zt}\in v_z+\bbR\alpha_t$ for any $z\in [x,y]$ and $t\in T$\\ such that $zt\in [x,y]$
\end{tabular}
\right\}.\]

We use this identification to give an upper bound to $\dim\Gamma_0(\calA,[x,y])^2=\dim V_{x,y} =d_{x,y}$. 

\begin{remark}
	The formula $q_{x,y}=c_{x,y}-\dim V_{x,y}$ already appeared in \cite{Dye2}. In his work Dyer proved this formula using different and more elementary techniques as moment graphs and Theorem \ref{KLconj} were still not available at that time. From this formula he derived a new proof of the positivity of the coefficient $q_{x,y}$ for arbitrary Coxeter groups (the positivity was originally proved by Tagawa in \cite{Tag}). 
	
	We remark that one could in principle rely  on Dyer's work and avoid the recourse to Theorem \ref{KLconj} also in the remainder of this work. 
	We choose to  reprove this formula in the moment graph setting since we hope (and believe) similar methods can be used to investigate also other coefficients of Kazhdan-Lusztig polynomials.
\end{remark}

In what follows we regard the moment graph solely as an undirected labelled graph, as the orientation of the arrows will not play any role since we are working with $\calA$ (cf. Example \ref{strsheaf}). 
We denote by $E_{x,y}$ the set of edges in $\calG|_{[x,y]}$. We denote an edge between $z$ and $w$ by $(z\rla w)$. For $e=(z\rla zt)\in E_{x,y}$ we define the function 
$\lambda_e(v): V_{x,y}\raw \bbR$ by $\lambda_e(v)=(v_z-v_{zt})/(\alpha_t)$. We will  sometimes denote the function $\lambda_e$ simply by $\lambda_{z,zt}$.

\begin{definition}
	We call \emph{diamond} any $4$-cycle in $\calG|_{[x,y]}$ consisting of four edges between four different vertices.
	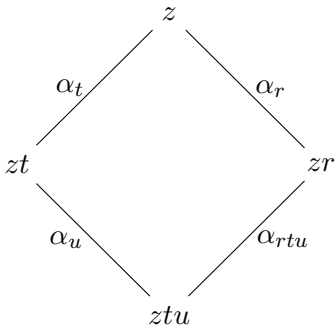
\begin{figure}[!h]
		\begin{center}
			\begin{tikzpicture}
			\node (a) at (0,2) {$z$};
			\node (b) at (-2,0) {$zt$};
			\node (c) at (2,0) {$zr$};
			\node (d) at (0,-2) {$ztu$};
			\path[-] 
			(a) edge node[left] {$\alpha_t$} (b)
			(c) edge node[right] {$\alpha_r$} (a)
			(b) edge node[left] {$\alpha_u$} (d)
			(c) edge node[right] {$\alpha_{rtu}$} (d);
			\end{tikzpicture}
		\end{center}
		\caption{A diamond}
		\label{figsq}
	\end{figure} 
\end{definition}

Notice that in every diamond the labels of a pair of adjacent edges are linearly independent while the span of all its labels has dimension $2$ (cf. \cite[Lemma 3.1]{Dye}).

The following key observation is inspired by \cite[Example 2.3]{BMP}.
Assume we have a diamond in $\calG|_{[x,y]}$ as in Figure \ref{figsq}
i.e. $z,zt,zr,ztu\in [x,y]$ with $t,r,u,rtu\in T$. 
Since $\alpha_t$ and $\alpha_r$ are linearly independent, for any $(v_z)\in V_{x,y}$ the intersection of the lines $v_{zt}+\bbR \alpha_t$ and $v_{zr}+\bbR \alpha_r$ is the point $v_z$. This means that $v_{zt}$ and $v_{zr}$ uniquely determine $v_z$ (and $v_{ztu}$). Equivalently, the real numbers  $\lambda_{z,zt}(v)$ and $\lambda_{z,zr}(v)$ uniquely determine $\lambda_{zt,ztu}(v)$ and $\lambda_{zr,ztu}(v)$.

\begin{definition}
Let $F$ be a subset of $E_{x,y}$. 
We say that $F$ is \emph{diamond closed} when for any $z,zt,zr,ztu$ forming a diamond as in Figure \ref{figsq}, if  $(z\rla zt), (z\rla zr)\in F$, then $(zt\rla ztu), (ztu\rla zr)\in F$.

We denote by $F^\diamond$ the smallest diamond closed subset of $E_{x,y}$ such that $F\cug F^\diamond$. We call $F^\diamond$ the \emph{diamond closure} of $F$.

We call a subset of edges $F$ \emph{diamond generating} if $F^\diamond=E_{x,y}$.
\end{definition}

The diamond closure $F^\diamond$ is well-defined for any $F$ and it can be simply obtained as follows. If $F\neq F^\diamond$, there exists a diamond $\calD$ in $\calG|_{[x,y]}$ such that two adjacent edges of $\calD$ are in $F$ but not all the edges of  $\calD$  are contained in $F$. Then we build a new set $F'=F\cup \{$edges of $\calD\}$. Clearly we have $F\subsetneqq F'\cug F^\diamond$, so we can replace $F$ with $F'$ and repeat this operation until we obtain $F=F^\diamond$.





\begin{definition}
We define 
\[g_{x,y}:=\min\{ |F| \mid F \text{ diamond generating subset of  }E_{x,y}\}.\] 
\end{definition}

In section \ref{secProof} we show that for simply-laced Weyl groups $g_{x,y}=d_{x,y}$, and doing so we assign a combinatorial meaning to the coefficient $d_{x,y}$. Here we show first in full generality one inequality.
\begin{prop}
Let $x,y\in W$ with $x<y$. We have $ d_{x,y}\leq g_{x,y}$.
\end{prop}
\begin{proof}

We need to show that for every diamond generating set $F$ we have $d_{x,y}\leq |F|$.

Assume $F$ is a diamond generating subset of $E_{x,y}$ and let $v\in V_{x,y}$.
We immediately see from the construction of $F^\diamond$ given above that
 the numbers  $\lambda_f(v)$ for every $f\in F^\diamond=E_{x,y}$ are uniquely determined by $(\lambda_e(f))_{e\in F}$. Moreover, recall that we have $v_y=0$. Thus  $v_z$ for every $z\in [x,y]$ is uniquely determined by $(\lambda_e(v))_{e\in F}$. 


Let $\bbR^F$ be a vector space of dimension $|F|$. Then from the discussion above it follows that the linear map
\[V_{x,y}\raw \bbR^F\]
\[ v\mapsto (\lambda_e(v))_{e\in F}\]
is injective. Hence $d_{x,y}=\dim V_{x,y}\leq |F|$.
\end{proof}

Furthermore, we can restrict ourselves to look at the Hasse diagram of $[x,y]$, i.e. it is enough to consider only edges $(z-zt)$ in $E_{x,y}$ with $|
\ell(zt)-\ell(z)|=1$. In fact, as it is shown in the the proof of \cite[Proposition 3.3]{Dye}, if $z\in W$ and $t\in A(z)$  are such that $\ell(zt)-\ell(z)>1$, then there exists a subgraph of $\calG_{[z,zt]}$ of the following form:
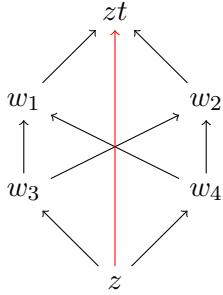
\begin{figure}[!ht]
\begin{center}
\begin{tikzpicture}[scale = 0.6]
\node (a) at (0,2) {$zt$};
\node (b) at (-2,0) {$w_1$};
\node (c) at (2,0) {$w_2$};
\node (d) at (-2,-2) {$w_3$};
\node (e) at (2,-2) {$w_4$};
\node (f) at (0,-4) {$z$};
\path[->] 
(f) edge (d)
(f) edge (e)
(d) edge (b) edge (c)
(e) edge (b) edge (c) 
(b) edge (a)
(c) edge (a);
\path[->, color = red]
(f) edge  (a); 
\end{tikzpicture}
\end{center}
\caption{A long arrow $(z-zt)$ in $\calG|_{[x,y]}$}
\label{exa}
\end{figure}

\begin{lemma}\label{shortedges}
If $F$ is a subset of edges such that $F^\diamond$ contains all the edges $(z-zt)$ in $E_{x,y}$ with $|\ell(zt)-\ell(z)|=1$ then $F$ is diamond generating.
\end{lemma}
\begin{proof}
If $F$ is not diamond generating, we can take an edge $(z-zt)$ such that $(z-zt)\not \in F^\diamond$ with $|\ell(z)-\ell(zt)|$  minimal amongst all the edges not in $F^\diamond$. Then using (for example) the diamond with edges $z,zt, w_1$ and $w_3$ in Figure \ref{exa} we see that $F^\diamond$ must also contain the edge $(z-zt)$.
\end{proof}

\begin{remark}[Upper bounds on $g_{x,y}$]\label{lowerbounds}
The set of edges 
\[F=\{(w-y)\mid w\in \coat(x,y)\}\]
 is a diamond generating set. 
 In fact, assume $z\in [x,y]$ with $\ell(y)-\ell(z)\geq 2$ and  consider an edge $(z-zt)$ with $zt\gtrdot z$. Then if we take any $w\in [x,y]$ such that $w\gtrdot zt$, we see by Lemma \ref{squares} that the 
 the elements in the interval $[z,w]$ form a diamond. Hence, by induction on $\ell(y)-\ell(z)$, we get $(z-zt)\in F^\diamond$ and, by Lemma \ref{shortedges}, it follows that $F$ is diamond generating.
 The same also holds for the set 
 \[F=\{(x-z)\mid z\in \at(x,y)\}.\]
We obtain $g_{x,y}\leq \min\{ |\at(x,y)|,|\coat(x,y)|\}$.
 

If $x=z_0\lessdot z_1\lessdot \ldots \lessdot z_{\ell(y)-\ell(x)}=y$ is a maximal chain from $x$ to $y$ then the set 
\[F = \{ (z_i -z_{i+1}) \mid 0\leq i \leq \ell(y)-\ell(x)-1\}\]
is diamond generating. This is an immediate consequence of the shellability of Bruhat intervals \cite{BjW}. It follows $g_{x,y}\leq \ell(y)-\ell(x)$.
\end{remark}

\begin{example}
Let $W$ be a Weyl group of type $A_3$ with simple reflections $S=\{s,t,u\}$. We draw the graph $\calG|_{[t,tsut]}$.

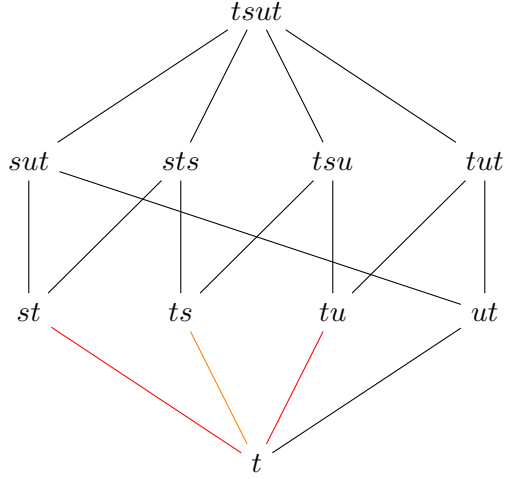
\begin{center}
\begin{tikzpicture}
\node (tsut) at (3,6) {$tsut$};
\node (sut) at (0,4) {$sut$};
\node (sts) at (2,4) {$sts$};
\node (tsu) at (4,4) {$tsu$};
\node (tut) at (6,4) {$tut$};
\node (st) at (0,2) {$st$};
\node (ts) at (2,2) {$ts$};
\node (tu) at (4,2) {$tu$};
\node (ut) at (6,2) {$ut$};
\node (t) at (3,0) {$t$};
\path[-] 
(t)  edge (ut)
(st) edge (sts) edge (sut)
(ut) edge (sut) edge (tut)
(ts) edge (sts) edge (tsu)
(tu) edge (tsu) edge (tut)
(tsut) edge (sut) edge (sts) edge (tsu) edge (tut) 
;
\path[-, color = red]
(t) edge (st) edge (tu);
\path[-, color = orange]
(ts) edge (t);

\end{tikzpicture}
\end{center}
If we take as $F_1$ the set $\{(t-st), (t-tu)\}$ then $F_1=F_1^\diamond$ as there is no diamond containing the two edges in $F_1$.

If we take as $F_2$ the set $\{(t-st), (t-tu), (t-ts)\}$ then $F_2$ is diamond generating. For example, we can look at the diamonds with vertices $t,st,ts,sts$ and $t,ts,tu,tsu$ to see that $(ts-sts)$ and $(ts-tsu)$ are in $F_2^\diamond$. Hence also $(sts-tsut)\in F_2^\diamond$, so $F_2^\diamond$ contains a maximal chain in $[x,y]$ and we can conclude by Remark \ref{lowerbounds}.

Since $d_{t,tsut}=3$, this also shows that $g_{x,y}=|F_2|=3$.
\end{example}

\section{The generalised lifting property}\label{secGLP}

In this section we assume that $W$ is finite and simply-laced, i.e. that $W$ is a Weyl group of type $A$, $D$ or $E$. Recall the partial order $\succ$ on $T$ from Section \ref{secCG}.

\begin{definition}\label{minrefdef}
Assume $\ell(y)>\ell(x)$. Then a minimal element $t\in AD(x,y)$ with respect to $\prec$ is called a \emph{minimal reflection} for $(x,y)$.
\end{definition}

\begin{theorem}[{\cite[Prop. 5.3]{TW}, \cite[Prop. 5.4]{CS}}]\label{GLT}
Let $x<y$ in $W$ and let $t$ be a minimal reflection for $(x,y)$. Then we have:
\begin{itemize} 
\item $x\leq yt\lessdot y$
\item$x\lessdot xt\leq y$.
\item $\displaystyle R_{x,y}(q) = (q-1) R_{x,yt}(q) +q R_{xt,yt}(q)$.
\end{itemize}
\end{theorem}

As a corollary, Lemma \ref{drec} immediately generalises to minimal reflections. If $x<y$ and $t$ is a minimal reflection for $(x,y)$, we have:
\begin{equation}\label{drect}
d_{x,y} = \begin{cases}
d_{x,yt} +1 & \text{ if }xt\not\leq yt \\
d_{x,yt}  & \text{ if }xt\leq yt. \\
\end{cases}
\end{equation}

\begin{example}
Assume $W$ is the symmetric group $S_n$. We write $x=(x(1)x(2)\ldots x(n))$ and $y=(y(1)y(2)\ldots y(n))$ for the corresponding permutations. Then the minimal reflections for $(x,y)$ are the transpositions $t=(i,j)$ with $i<j$ such that $[i,j]$ is a minimal interval (with respect to the inclusion order) satisfying $y(i)>y(j)$ and $x(i)<x(j)$.
\end{example}

Let $x<y$ and let $t$ be a minimal reflection for $(x,y)$.
As a corollary of Theorem \ref{GLT}, the reflection $t$ is also a minimal element of the set $\at^T(x,y)\cap \coat^T(x,y)$.


\begin{lemma}\label{onenotsmaller}
Assume $x<y$ with $\ell(y)-\ell(x)\geq 2$ and let $t$ be a minimal reflection for $(x,y)$. Then there exists $r\in \at^T(x,y)\cup \coat^T(x,y)$ such that $r\not \preceq t$.
\end{lemma}
\begin{proof}
Take $u$ a minimal reflection for $(x,yt)$. 
Clearly $u\neq t$. If $u\not\prec t$, then the claim follows because $u\in  \at^T(x,y)$.

So we can assume $u\prec t$. 
We have $y\gtrdot yt \gtrdot ytu$ and, by applying Lemma \ref{squares} to $[ytu,y]$, we have $yu\in [ytu,y]$ or $ytut\in [ytu,y]$, hence $u\in D(y)$ or $tut\in D(y)$. But if $u\in D(y)$, then also $u\in AD(x,y)$, contradicting the minimality of $t$. It remains to consider the case $tut\in D(y)$ with $t$ and $u$ not commuting, hence $tut=utu$.
\begin{center}
\begin{tikzpicture}[scale=0.7]
\node (a) at (0,2) {$y$};
\node (b) at (-2,0) {$yt$};
\node (c) at (2,0) {$ytut$};
\node (d) at (0,-2) {$ytu$};
\path[-] 
(a) edge node[left] {$\alpha_t$} (b)
(c) edge node[right] {$\alpha_{tut}$} (a)
(b) edge node[left] {$\alpha_{u}$} (d)
(c) edge node[right] {$\alpha_{t}$} (d);
\end{tikzpicture}
\end{center}

Since $t\in D(ytut)$ and $u=(tut)t(tut)\in A(y)$, we can apply Lemma \ref{23}.i) for $x=ytut$, $t_1=t$ and $t_2=tut$ to deduce $t\in D(tut)$. Hence, by \eqref{Dimplies}, $t\prec tut$ and we conclude since $tut\in \coat^T(x,y)$.
\end{proof}

\begin{remark}
	We can slightly strengthen Lemma \ref{onenotsmaller} for groups of type $A$, i.e. for the symmetric group $S_n$. As in the proof of Lemma \ref{onenotsmaller} let $x<y$, let $t$ be a minimal reflection for $(x,y)$ and $u$ be a minimal reflection for $(x,yt)$. If $u\prec t$, then the same proof as above shows that there exists $r\in \coat^T(x,y)$ with $r\not\preceq t$. 
	
	If $u\not\preceq t$, then, by Lemma \ref{squares}, either $u\in \coat^T(x,y)$, or $u$ and $t$ do not commute and $tut\in \coat^T(x,y)$. We claim that in the latter case we have $tut\not\prec t$. In fact, we have $t(\alpha_u)=\alpha_u\pm \alpha_t$ and, since $\alpha_u\not\preceq\alpha_t$, $t(\alpha_u)$ must be a positive root, namely $t(\alpha_u)=\alpha_{tut}$. If $\alpha_u+ \alpha_t=\alpha_{tut}$ then $tut\succ t$. Assume now $\alpha_u- \alpha_t=\alpha_{tut}$. Notice that if $u$ is the transposition $(i,j)\in S_n$ and $t=(i',j')\in S_n$ then $\alpha_u-\alpha_t$ is a positive root if and only if $i=i'$ and $j'<j$ or $i'>i$ and $j=j'$, so it is easy to see that $\alpha_{tut}=\alpha_u-\alpha_t\not\preceq \alpha_t$. 
	
	Hence, in type $A$ we can always find $r\in \coat^T(x,y)$ such that $r\not\preceq t$. By a symmetric argument, we can also always find $r'\in \at^T(x,y)$ such that $r'\not\preceq t$.
\end{remark}

\begin{lemma}\label{newmin}
Assume $x<y$ with $\ell(y)-\ell(x)\geq 2$ and let $t$ be a minimal reflection for $(x,y)$. Then there exists $z\in \at(x,y)$ such that $t$ is a minimal reflection for $(z,y)$ or $w\in \coat(x,y)$ such that $t$ is a minimal reflection for $(x,w)$.
\end{lemma}
\begin{proof}
From Lemma \ref{onenotsmaller}, the set 
\[U := \{ r \in \at^T(x,y)\cup \coat^T(x,y)\mid r\not\preceq t\}\]
is not empty. Let $r$ be a maximal element in $U$ with respect to $\succ$. We can assume $r\in \coat^T(x,y)$ as the case $r\in \at^T(x,y)$ is completely symmetric.

Since $r\not\prec t$, by \eqref{Dimplies},  we have $r\in A(t)$. Then, by Lemma \ref{23}.i) for $x=y$, $t_1=r$ and $t_2=t$, since $r\in D(y)$ we have $trt\in D(yt)$. In other words, we have $\ell(yrt)<\ell(yt)=\ell(yr)$, 
 thus also $t\in D(yr)$. Since $t\in A(x)$, we obtain $t\in AD(x,yr)$.  
 
Assume that $t$ is not minimal in $AD(x,yr)$, thus there exists a minimal reflection $u$ for $(x,yr)$ with $u\prec t$. Applying Lemma \ref{squares} to the interval $[yru,y]$ we get $yu\in [yru,y]$ or $yrur\in [yru,y]$. But we cannot have $u\in D(y)$ since $t$ is minimal in $AD(x,y)$, thus we deduce that $yrur\in [yru,y]$ and that $r$ and $u$ do not commute.

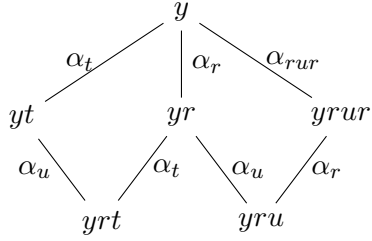
\begin{figure}[h]
	\begin{center}
		\begin{tikzpicture}[scale=0.7]
		\node (a) at (0,2) {$y$};
		\node (b) at (-3,0) {$yt$};
		\node (c) at (0,0) {$yr$};
		\node (d) at (-1.5,-2) {$yrt$};
		\node (e) at (3,0) {$yrur$};
		\node (f) at (1.5,-2) {$yru$};
		\path[-] 
		(a) edge node[left] {$\alpha_t$} (b)
		(c) edge node[right] {$\alpha_{r}$} (a)
		(b) edge node[left] {$\alpha_{u}$} (d)
		(c) edge node[right] {$\alpha_{t}$} (d)
		(c) edge node[right] {$\alpha_u$} (f)
		(a) edge node[right] {$\alpha_{rur}$} (e)
		(e) edge node[right] {$\alpha_r$} (f);
		\end{tikzpicture}
	\end{center}
\end{figure}

Now we have $r\in D(yrur)$ and $u=(rur)r(rur)\in A(y)$. Lemma \ref{23}.i) for $x=yrur$, $t_1=r$ and $t_2=rur$ implies that $r\in D(rur)$ and in particular, by \eqref{Dimplies}, we have $r\prec rur$, contradicting the maximality of $r$ in $U$.
\end{proof}

\begin{prop}\label{tchain}
Assume $x<y$ and let $t$ be a minimal reflection for $(x,y)$. Then there exists a maximal chain $x=z_0\lessdot z_1\lessdot \ldots \lessdot z_{\ell(y)-\ell(x)-1} = yt$ such that $z_it\in [x,y]$ and $z_it \gtrdot z_i$ for all $i$.
\end{prop}
\begin{proof}
We prove the claim by induction on $\ell(y)-\ell(x)$. The case $\ell(y)-\ell(x)=1$ is trivial, so we can assume $\ell(y)-\ell(x)\geq 2$.

From Lemma \ref{newmin} we can assume that there exists $r\in \coat^T(x,y)$ such that $t$ is a  minimal reflection for $(x,yr)$ (the case $r\in \at^T(x,y)$ is completely symmetric).

By induction there exists a chain $x=z_0\lessdot z_1\lessdot \ldots \lessdot z_{\ell(y)-\ell(x)-2} = yrt$ with $z_it\in [x,yr]$ and $z_it\gtrdot z_i$ for all $i\leq \ell(y)-\ell(x)-2$. 
Since $yrt=yt(trt)$ we have $yrt\lessdot yt$, so we conclude by setting $z_{\ell(y)-\ell(x)-1} = yt$.
\end{proof}

\section{The coefficient  of \texorpdfstring{$q$}{q} in finite simply-laced type}\label{secProof}

We are ready to finally put together the results of Section \ref{secqMG} and \ref{secGLP} to obtain the main theorem of this paper. 
 
\begin{theorem}
Let $W$ be a Weyl group of type ADE. Then for any $x,y\in W$ with $x<y$ we have $d_{x,y}=g_{x,y}$. 
\end{theorem}
\begin{proof}
If $\ell(y)-\ell(x)=1$ the claim is clear since $d_{x,y}=g_{x,y}=1$. 
Let $x,y\in W$ with  $\ell(y)-\ell(x)>1$. By induction, assume the claim for any $x',y'\in W$ with $\ell(y')-\ell(x')<\ell(y)-\ell(x)$.

Let $t$ be a minimal reflection for $(x,y)$, hence $x\lessdot xt< y$ and $x<yt\lessdot y$. We divide the proof into two cases as in Equation \eqref{drect}. 

\textbf{Case 1:} Assume $xt\leq yt$.  We have $g_{x,y}\geq d_{x,y}=d_{x,yt}=g_{x,yt}$. Then it is enough to show $g_{x,yt}\geq g_{x,y}$, or that any diamond generating subset $F\cug E_{x,yt}$ is also diamond generating as a subset of $E_{x,y}$. Since $xt\leq yt$ the edge $(x\rla xt)$ belongs to $F^\diamond$. Let $x=z_0\lessdot z_1\lessdot \ldots \lessdot z_{\ell(y)-\ell(x)-1}=yt$ be a maximal chain  with $z_it\gtrdot z_i$ and $z_it\in [x,y]$ as in Lemma \ref{tchain}. Then since $F^\diamond$ contains $(x\rla xt)$ and $(z_i\rla z_{i+1})$ for any $0\leq i\leq \ell(y)-\ell(x)-2$,  it must also contain the edges $(z_i-z_it)$ and $(z_it\rla z_{i+1}t)$  for any $i$ (see Figure \ref{ladder}). In particular,  $(yt\rla y)\in F^\diamond$, so $F^\diamond$ contains a maximal chain from $x$ to $y$ and, as in Remark \ref{lowerbounds}, this implies that $F$ is a diamond generating subset of $E_{x,y}$.
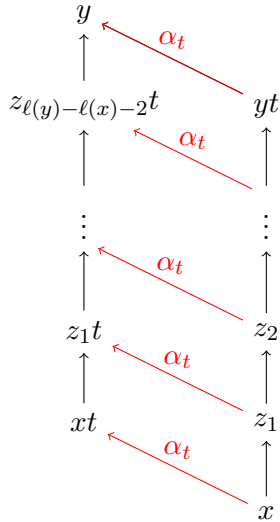
\begin{figure}[h]
\begin{center}
\begin{tikzpicture}[scale = 0.6]
\node (a) at (-2,7) {$y$};
\node (b) at (-2,0) {$z_1t$};
\node (c) at (2,0) {$z_2$};
\node (b0) at (-2,-2) {$xt$};
\node (c0) at (2,-2) {$z_1$};
\node (b2) at (-2, 2) {$\;$};
\node (c2) at (2, 2) {$\; $};
\node  at (-2, 2.5) {$\vdots$};
\node  at (2, 2.5) {$\vdots$};
\node (b22) at (-2, 3) {$\;$};
\node (c22) at (2, 3) {$\;$};
\node (b3) at (-2, 5) {$z_{\ell(y)-\ell(x)-2}t$};
\node (c3) at (2, 5) {$yt$};
\node (f) at (2,-4) {$x$};
\path[->] 
(f) edge (c0)
(c0) edge (c)
(c) edge (c2)
(c22) edge (c3) 
(c3) edge (a)
(b0) edge (b)
(b) edge (b2)
(b22) edge (b3)
(b3) edge (a);
\path[->, color = red]
(c0) edge node[above] {$\alpha_t$} (b)
(c) edge  node[above] {$\alpha_t$} (b2)
(c22) edge node[above] {$\alpha_t$} (b3)
(c3) edge node[above] {$\alpha_t$} (a)
(f) edge node[above] {$\alpha_t$} (b0);
\end{tikzpicture}
\end{center}
\caption{A ``ladder'' between $x$ and $y$}
\label{ladder}
\end{figure}

\textbf{Case 2:} Assume $xt\not \leq yt$. We have $g_{x,y}\geq d_{x,y}=d_{x,yt}+1=g_{x,yt}+1$. It is enough to show that $g_{x,yt}+1\geq g_{x,y}$. Let $F$ be any diamond generating subset of $E_{x,yt}$. Let $F'=F \cup \{x-xt\}$. Then the same argument of the previous case shows that $F'$ is a diamond generating subset of $E_{x,y}$. The claim follows.   
\end{proof}

As a byproduct of the proof we also obtain that there always exists a diamond generating subset $F\cug E_{x,y}$ such that for any $(z\rla w)\in F$ we have $|\ell(z)-\ell(w)|=1$. Hence to compute $g_{x,y}$ in type $ADE$ it is enough to look at the Hasse diagram of $[x,y]$.

\begin{cor}
The coefficient of $q$ of Kazhdan-Lusztig polynomials in type ADE is a combinatorial invariant, i.e. $q_{x,y}$ depends only on the poset structure of the Bruhat interval $[x,y]$. 

Moreover, the coefficient $q_{x,y}$ can be explicitly computed using the formula 
\[q_{x,y} = c_{x,y} - g_{x,y}.\]
\end{cor} 

\section*{Acknowledgements}

This paper was born as a consequence of many insightful discussions with Geordie Williamson, whom I would like to warmly thank.
I would also like to thank Ben McDonnell and Lars Thorge Jensen for their comments on a preliminary version of this paper.

\bibliographystyle{alpha}

\Address
\end{document}